\documentclass{amsart}
\usepackage{url}
\usepackage{mathrsfs}
\usepackage{enumitem}   % 允许自定义列表格式，特别是[Step 1:]这种格式
\usepackage{mathtools}  % 扩展了amsmath的功能
\usepackage{amssymb, graphicx, amsrefs}
\usepackage{xcolor}
\usepackage{svg}
\usepackage{booktabs}
\usepackage{siunitx}
\usepackage{color}%cyan,red;magenta,green;yellow,blue
\usepackage{graphicx}
\usepackage{float} 
\usepackage{subfigure}
\usepackage{bbm}
\usepackage{hyperref}
\usepackage{epstopdf}
\usepackage{fancyhdr} 
\usepackage{cases}
\copyrightinfo{2022}{Jin Sun}

\newtheorem{theorem}{Theorem}[]
\newtheorem{lemma}[theorem]{Lemma}

\theoremstyle{definition}
\newtheorem{definition}[theorem]{Definition}

\theoremstyle{remark}
\newtheorem{remark}[theorem]{Remark}

\numberwithin{equation}{section}

\newcommand{\partiald}[2]{  \frac{\partial #1 }{\partial #2}}

\newcommand{\derivative}[2]{ \frac{\mathrm{d} #1 }{\mathrm{d} #2}}
\newcommand{\bracket}[1]{\left( #1\right)}
\newcommand{\as}[1]{\left\langle #1\right\rangle}
\newcommand{\av}[1]{\left\vert #1\right\vert}

\newcommand{\R}{\mathbb{R}}

\providecommand{\eat}[1]{}

\newcommand{\Hm}[1]{\leavevmode{\marginpar{\tiny%
			$\hbox to 0mm{\hspace*{-0.5mm}$\leftarrow$\hss}%
			\vcenter{\vrule depth 0.1mm height 0.1mm width \the\marginparwidth}%
			\hbox to 0mm{\hss$\rightarrow$\hspace*{-0.5mm}}$\\\relax\raggedright
			#1}}}

\begin{document}
	
 %   \title[Fundamental Gap and Heat Kernels on Gaussian Spaces]{Sharp Fundamental Gap Estimate and Log-concavity of Heat Kernels on Convex Domains in Gaussian Spaces}

        \title[Fundamental Gap  on Gaussian Spaces]{Sharp Fundamental Gap Estimate on Convex Domains in Gaussian Spaces}
    
    \author{Jin Sun}
    \address{Jin Sun, School of Mathematical Sciences, Fudan University, 200433, Shanghai, China}
    \email{jsun22@m.fudan.edu.cn}

    \author{Kui Wang}
    \address{Kui Wang, School of Mathematical Sciences, Soochow University, 215006, Suzhou, China}
    \email{kuiwang@suda.edu.cn}
    
	
	%	\date{\today}
	%	\thanks{The author is supported by Shanghai Science and Technology Program [Project No. 22JC1400100]}
	
\begin{abstract}
We prove a sharp lower bound for the fundamental gap on convex domains in Gaussian spaces, the difference between the first two eigenvalues of the Ornstein-Uhlenbeck operator with Dirichlet boundary conditions. Our main result establishes that the gap is bounded below by the gap of the corresponding one-dimensional model, confirming the Gaussian analogue of the fundamental gap conjecture. Furthermore, we demonstrate that the normalized gap of the one-dimensional model is monotonically increasing with the diameter and prove the sharpness of our estimate. Beyond the fundamental gap, we also establish improved log-concavity properties for the Dirichlet heat kernel on convex domains in Gaussian spaces. Our work on Gaussian spaces complements the existing results of Andrews and Clutterbuck \cite{andrews2011} and Ni \cite{ni2013} for Euclidean domains, as well as the work of Seto, Wang, and Wei \cite{SWW19} for spherical domains.
\end{abstract} 
   
\maketitle
	
\section{Introduction}

The fundamental gap problem, concerning the difference $\lambda_2 - \lambda_1$ between the first two Dirichlet eigenvalues of the Laplacian, has been a central topic in spectral geometry for several decades. For convex domains $\Omega \subset \mathbb{R}^n$ with diameter $D$, it was conjectured independently by van den Berg \cite{vandenberg1983}, Ashbaugh-Benguria \cite{ashbaugh1989}, and Yau \cite{yau1986} that
\begin{align*}
 \Gamma(\Omega):= \lambda_2(\Omega) - \lambda_1(\Omega) \geq \frac{3\pi^2}{D^2}.
\end{align*}
After significant partial results by Singer-Wong-Yau-Yau \cite{singer1985} establishing $\Gamma(\Omega) \geq \pi^2/(4D^2)$ and Yu-Zhong \cite{yu1986} improving this to $\pi^2/D^2$, the conjecture was completely resolved by Andrews-Clutterbuck \cite{andrews2011} by establishing sharp log-concavity estimates for the first eigenfunction; see also the survey paper \cite{And15}. 
For convex domains in $\mathbb{S}^n$, sharp fundamental gap estimates have been established in various settings by Seto-Wang-Wei \cite{SWW19}, He-Wang-Zhang \cite{HWZ20}, Dai-Seto-Wei \cite{DSW21}. In hyperbolic space,  Bourni et al. \cite{BCNSWW22} showed that the fundamental gap can vanish on certain convex domains. For recent advances on fundamental gap estimates for convex domains on positively curved surfaces, we refer to \cite{CWY25} and \cite{KNTW25}, as well as results related to Robin eigenvalues \cites{ACH20, ACH21}. 
    
In addition to these developments in spectral geometry, log-concavity properties of heat kernels on convex domains have played a fundamental role in probability theory and analysis. The foundational result is the classic theorem of Brascamp and Lieb \cite{BrascampLieb1976}, who proved in 1976 that for the heat equation $\partial_t u = \Delta u - Vu$ on a convex domain $\Omega \subset \mathbb{R}^n$ with Dirichlet boundary conditions and convex potential $V$, the heat kernel $p(x,y,t)$ is log-concave in $x$ for each fixed $(y,t)$. An improved log-concavity estimate for the heat kernel was later obtained by Ni \cite{ni2013} via the elliptic maximum principle.

In parallel, the spectral theory of the Ornstein-Uhlenbeck operator has emerged as a fundamental area connecting analysis, probability, and geometry. Consider the Gaussian measure 
    \begin{equation*}
        d\mu = (2\pi)^{-n/2}e^{-|x|^2/2}dx.
    \end{equation*} 
The Ornstein-Uhlenbeck operator
    $$
    L_\mu u:= e^{|x|^2/2}\operatorname{div}(e^{-|x|^2/2}\nabla u)= \Delta u - \langle x, \nabla u \rangle
    $$
arises naturally as the infinitesimal generator of the Ornstein-Uhlenbeck process and plays a central role in Gaussian harmonic analysis. While the spectral properties of $L_\mu$ on $\mathbb{R}^n$ are well-understood through Hermite polynomial expansions \cite{hermite1864}, the Dirichlet problem on bounded convex domains presents significant new challenges.

For a bounded convex domain $\Omega \subset \mathbb{R}^n$, the Gaussian principal frequency is defined as
    $$
    \lambda_\mu(\Omega) = \inf\left\{ \frac{\int_\Omega |\nabla v|^2 d\mu}{\int_\Omega v^2 d\mu} : v \in W_0^{1,2}(\Omega, \mu), v \not\equiv 0 \right\}.
    $$
This corresponds to the first eigenvalue of the boundary value problem
    \begin{equation}\label{eq:eigenvalue}
    	\begin{cases}
    		L_\mu u = -\lambda_\mu u, & \text{ in } \Omega, \\
          u > 0 & \text{ in } \Omega, \\
    		u = 0 & \text{ on } \partial\Omega.
    	\end{cases}
    \end{equation}
Recent advances have established fundamental properties of $\lambda_\mu$. Ehrhard \cite{ehrhard1984} proved the Gaussian Faber-Krahn inequality, showing that half-spaces minimize $\lambda_\mu$ among sets of fixed Gaussian measure. Carlen-Kerce \cite{carlen2001} characterized the equality cases. Colesanti-Francini-Livshyts-Salani \cite{colesanti2024} showed that the first Dirichlet eigenfunction corresponding to $\lambda_\mu$ is log-concave, and this result was subsequently strengthened to strong log-concavity by Qin \cite{Qin2025}. However, the sharp fundamental gap estimate in Gaussian spaces and the log-concavity properties of the Gaussian heat kernel itself remained open questions.

Building on the approach of Andrews-Clutterbuck \cite{andrews2011} and Ni's elliptic method \cite{ni2013}, we establish the following sharp estimate:
    \begin{theorem}\label{main_theorem}
    	Let $\Omega \subset \mathbb{R}^n$ be a strictly convex domain with diameter $D$, and let $\lambda_{1,\mu}$ and $\lambda_{2,\mu}$ be the first two eigenvalues of the Ornstein-Uhlenbeck operator $L_\mu$ with Dirichlet boundary conditions. Then
    	\begin{equation}\label{fundamental}
    	    \lambda_{2,\mu} - \lambda_{1,\mu}\geq \bar{\lambda}_2(D) - \bar{\lambda}_1(D),
    	\end{equation}
    	where $\bar{\lambda}_i(D)$ are the eigenvalues of the one-dimensional model operator
    	\begin{equation}\label{one_dim_model}
    	    -\derivative{^2}{s^2}+s\derivative{}{s}
    	\end{equation}
    	on the interval $(-D/2, D/2)$ with Dirichlet boundary conditions. Furthermore,       for all $D>0$, we have
    	$$
    	\bar{\lambda}_2(D) - \bar{\lambda}_1(D)> \frac{3\pi^2}{D^2}.
    	$$
  
    \end{theorem}
    
The key to our proof is using the gauge transformation to convert the Ornstein-Uhlenbeck operator into a Schr\"odinger operator with a convex potential $\av{x}^2$, establishing the log-concavity of the first eigenfunction for the Schr\"odinger operator, and extending the elliptic method of Ni \cite{ni2013} for the Schr\"odinger operator on convex domains.

Moreover, we establish improved log-concavity properties for the Gaussian heat kernel on convex domains. We consider the Gaussian heat kernel defined as follows.
    \begin{definition}
        The function $H_G(x,y,t)$ defined on $\Omega\times\Omega\times\R_{+}$ is called the Gaussian heat kernel on $\Omega$ if $H_G(x,y,t)$ is the fundamental solution of 
        \begin{equation}\label{eq:GaussHeatKernel}
    	\begin{cases}
    		\partial_t u(\cdot,y,t)=L_\mu u(\cdot,y,t), & \text{ in } \Omega, \\
    		u(\cdot,y,t) = 0, & \text{ on } \partial\Omega,
    	\end{cases}
        \end{equation}
    \end{definition}
    
Motivated by the classical results of Brascamp-Lieb \cite{BrascampLieb1976} and the work of Ni \cite{ni2013} for vector fields, we establish improved log-concavity estimates for the Gaussian heat kernel on convex domains with Dirichlet boundary data. Our main result on log-concavity of the Gaussian heat kernel is the following.

    \begin{theorem}\label{theorem-heatkernel}
        Let $\Omega \subset \mathbb{R}^n$ be a strictly convex domain with diameter $D$. If $H_G(x,y,t)$ is the Gaussian heat kernel on $\Omega$
        and if $\overline{H}_G(s,t)$ is the fundamental solution of $\partial_t-\partial_s^2+s\partial_s$ centered at $0$ with Dirichlet boundary condition on $[-D/2,D/2]$, then for any $t>0, \ x\neq y\in\Omega$ and any fixed point $z\in\Omega$,
        \begin{equation}\label{log-concavity-heatkernel}
            -\left\langle \nabla_y \log H_G(z,y,t) - \nabla_x \log H_G(z,x,t), \frac{y-x}{|y-x|} \right\rangle \geq -2\left(\log\overline{H}_G\right)^\prime\left(\frac{|y-x|}{2},t\right). 
        \end{equation}
    \end{theorem}
    \begin{remark}
       Estimates \eqref{fundamental} and \eqref{log-concavity-heatkernel} remain valid for Schr\"odinger operators of the form $-L_\mu + V$ with any convex potential  $V$.
    \end{remark}

The proof of Theorem~\ref{theorem-heatkernel} is based on on two key elements. Firstly, we employ the calculus of variations to establish the log-concavity of the one-dimensional Gaussian heat kernel. Secondly, we adapt Ni's elliptic method \cite{ni2013} to the Gaussian setting, which yields a quantitative comparison between the gradient of the logarithm of heat kernel and the one-dimensional model. As shown in Lemma \ref{log-concavity-1-dim} (Section \ref{sect5}), we have $\left(\log\overline{H}_G\right)^\prime(s,t)<0$ for $s>0$. Thus, this estimate \eqref{log-concavity-heatkernel}, in turn, provides an improved log-concavity estimate.

The paper is organized as follows. Section \ref{sect2} establishes preliminaries on Gaussian spaces and transforms the Ornstein-Uhlenbeck operator into a Schr\"odinger operator. Section \ref{sect3}  uses the maximum principle method to derive the lower bound of the fundamental gap of the Ornstein-Uhlenbeck operator. Section \ref{sect4}  analyzes the one-dimensional model and proves crucial monotonicity properties as well as the sharpness of the gap. Section \ref{sect5} proves the improved log-concavity of the Gaussian heat kernel on convex domains. Finally, the appendix provides the graph and table of the fundamental gap of the one-dimensional model as a function of diameter $D$.

    \section{Preliminaries}\label{sect2}
    Let $\mathbb{R}^n$ be equipped with the standard Gaussian measure
    $d\mu$.
    We call the space $(\mathbb{R}^n, d\mu)$ the \emph{Gaussian space}. For a measurable set $A \subset \mathbb{R}^n$, its Gaussian measure is denoted by $\mu(A)$. The Ornstein-Uhlenbeck operator $L_\mu$ is self-adjoint with respect to the Gaussian measure. Indeed, for $\phi, \psi \in C_c^\infty(\Omega)$, integration by parts yields
    $$
    \int_\Omega \phi L_\mu \psi \, d\mu = -\int_\Omega \langle \nabla \phi, \nabla \psi \rangle \, d\mu = \int_\Omega \psi L_\mu \phi \, d\mu.
    $$

    For a bounded domain $\Omega \subset \mathbb{R}^n$ with Lipschitz boundary and diameter $D$, we consider the weighted Sobolev space
    \begin{align*}
        W_0^{1,2}(\Omega, \mu) = \left\{ u \in L^2(\Omega, \mu) : |\nabla u| \in L^2(\Omega, \mu), \, u|_{\partial\Omega} = 0 \right\}.
    \end{align*}
    By standard variational arguments \cite{evans2010},
    the eigenvalues $\{\lambda_{i,\mu}\}_{i=1}^\infty$ of the Ornstein-Uhlenbeck operator on $\Omega$ are given by
    \begin{equation}
        \lambda_i(\Omega) = \inf_{\substack{E_i\subset C_0^\infty(\Omega)\\ E_i:\ i\text{-dim space}}}\sup_{v\in E_i} R_\mu(v),
    \end{equation}
    where the Rayleigh quotient is given by
    \begin{equation*}
        R_\mu(v) = \frac{\int_\Omega |\nabla v|^2 \, d\mu}{\int_\Omega v^2 \, d\mu}.
    \end{equation*}

    For general weighted spaces $(\R^n, e^{-f}dx)$,  the eigenvalue problem with a convex potential $V$ can be described as follows:
    \begin{equation*}
        \begin{cases} 
            -\Delta u + \langle \nabla f, \nabla u \rangle + Vu(x) = \lambda u(x), & x \in \Omega, \\ 
            u(x) = 0, & x \in \partial \Omega.
        \end{cases}
    \end{equation*}
    By applying the gauge transformation $u(x)=v(x)e^{f(x)/2}$, one removes the drift term and obtains an equivalent eigenvalue problem for $v$, given by
    \begin{equation*}
        \begin{cases} 
            -\Delta v(x) + (\frac{1}{4}\av{\nabla f}^2-\frac{1}{2}\Delta f) v(x) + Vv(x) = \lambda v(x), & x \in \Omega, \\ 
            v(x) = 0, & x \in \partial \Omega.
        \end{cases}
    \end{equation*}
    It is noteworthy that for any harmonic function or any quadratic function $f(x)$, the function $\frac{1}{4}\av{\nabla f}^2-\frac{1}{2}\Delta f$ is always convex in $\Omega$. Therefore, the original eigenvalue problem is equivalent to a standard Laplacian eigenvalue problem with a convex potential.
    In particular, after the gauge transformation $v(x)=u(x)e^{-|x|^2/4}$, the problem of the operator $-L_\mu+V$ is equivalent to the following eigenvalue problem:
    \begin{equation}\label{equivalent_ep}
        \begin{cases} 
            -\Delta v(x) + \frac{1}{4}|x|^{2}v(x) + Vv(x) = (\lambda+\frac{n}{2}) v(x), & x \in \Omega, \\ 
            v(x) = 0, & x \in \partial \Omega.
        \end{cases}
    \end{equation}
    This means that to get the lower bound of the fundamental gap of convex domains in Gaussian spaces, we only need to consider the fundamental gap of \eqref{equivalent_ep}. For convenience, we consider the convex domains in Euclidean spaces $\R^n$ and consider the Schr\"odinger operator
    \begin{equation*}
        \Delta_G:=\Delta-\frac{1}{4}\av{x}^2.
    \end{equation*}
    The eigenvalues of the operator $\Delta_G$ of $\Omega$ are denoted by $\{\lambda_{i,G}\}_{i=1}^\infty$. From \eqref{equivalent_ep}, we know that for any $i\geq 1$,
    \begin{equation}\label{difference_eigenvalue}
        \lambda_{i,G} = \lambda_{i,\mu} + \frac{n}{2}.
    \end{equation}
    Observe that for any convex potential $V$,
    \begin{equation*}
        \left(\nabla (\frac{1}{4}\av{y}^2+V(y))-\nabla(\frac{1}{4}\av{x}^2+V(x))\right)\cdot\frac{y-x}{\av{y-x}}\geq 2\derivative{(\frac{1}{4}s^2)}{s}\Bigg|_{s=\av{y-x}/2}.
    \end{equation*}
    Then, as a direct conclusion of the gauge transformation and Theorem 1.5 in \cite{andrews2011}, we have the following log-concavity comparison theorem.
    
    \begin{theorem}\label{theorem3}
        Let $\Omega$ be a strictly convex bounded domain in $\mathbb{R}^n$ with diameter $D$. Let $\varphi_{1}(x)$ be the first positive eigenfunction on $\Omega$ given by \eqref{eq:eigenvalue}. Let $\overline\varphi_{1}(s)$ be the first positive Dirichlet eigenfunction of the one-dimensional model operator \eqref{one_dim_model} on $(-D/2, D/2)$. Then for all $x\neq y\in\Omega$,
        \begin{equation}\label{log-concavity}
            -\left\langle \nabla \log \varphi_{1}(y) - \nabla \log \varphi_{1}(x), \frac{y-x}{|y-x|} \right\rangle \geq -2\left(\log\overline\varphi_{1}\right)^\prime\left(\frac{|y-x|}{2}\right). 
        \end{equation}
    \end{theorem}
    \begin{remark}
        The inequality \eqref{log-concavity} can also follow from Theorem \ref{theorem-heatkernel}, since $e^{\lambda_1 t}H_G(z,x,t)\rightarrow\varphi_{1}(z)\varphi_{1}(x)$ as $t\rightarrow+\infty$.
    \end{remark}
    In fact, we can show that the first positive Dirichlet eigenfunction $\overline\varphi_{1}(s)$ is concave, and is therefore log-concave.
    \begin{lemma}\label{lemma_concave}
        Let $\overline\varphi_{1}(s)$ be the first positive Dirichlet eigenfunction of the one-dimensional model operator \eqref{one_dim_model} on $(-D/2, D/2)$. Then for all $s\in (-\frac{D}{2}, \frac{D}{2})$,
        \begin{equation*}
            \overline\varphi_{1}^{\prime\prime}(s)<0.
        \end{equation*}
    \end{lemma}
    \begin{proof}
        Since $\overline\varphi_{1}(s)$ is the minimizer of
        \begin{equation*}
            R_\mu(v) =\frac{\int_{-\frac{D}{2}}^{\frac{D}{2}} |\nabla v|^2 d\mu}{\int_{-\frac{D}{2}}^{\frac{D}{2}} v^2 d\mu}, 
        \end{equation*}
    where $ v \in W_0^{1,2}((-\frac{D}{2}, \frac{D}{2}), \mu)\setminus\{0\}$ and $d\mu=(2\pi)^{-1/2}e^{-s^2/2}ds$,
    we can see that $\overline\varphi_{1}(s)$ is an even and smooth function, and it does not vanish in the interior. Without loss of generality, we assume that 
        \begin{equation*}
            \overline\varphi_{1}(0)=1,\ \overline\varphi_{1}^{\prime}(0)=0, \ \overline\varphi_{1}(s)>0,\ \forall s\in(-\frac{D}{2}, \frac{D}{2}).
        \end{equation*}
        By equation \eqref{one_dim_model}, $\overline\varphi_{1}(s)$ satisfies 
        \begin{equation*}
            -\overline{\varphi}_1^{\prime\prime}(s)+ s\overline{\varphi}_1^\prime(s)=\overline{\lambda}_1\overline{\varphi}_1(s).
        \end{equation*}
        Thus, $\overline{\varphi}_1^{\prime\prime}(0)=-\overline{\lambda}_1<0$. Assume that $\overline{\varphi}_1^{\prime\prime}$ has some zeros in the interior and
        \begin{equation*}
            s_0:=\inf\left\{s\in(0, \frac{D}{2}):\overline{\varphi}_1^{\prime\prime}(s)=0\right\}.
        \end{equation*}
        Then $s_0>0$, and for all $0<s<s_0$, $\overline{\varphi}_1^{\prime\prime}(s)<0$. Hence $\overline{\varphi}_1^\prime$ is decreasing in $(0,s_0)$ and $\overline{\varphi}_1^\prime<0$ in $(0,s_0)$. But at $s_0$, 
        \begin{equation*}
        s_0\overline{\varphi}_1^\prime(s_0)=\overline{\lambda}_1\overline{\varphi}_1(s_0)>0,
        \end{equation*}
        a contradiction. So $\overline{\varphi}_1^{\prime\prime}$ does not vanish in the interior, which means $\overline{\varphi}_1^{\prime\prime}$ is always negative on $(-{D}/{2}, {D}/{2})$.
    \end{proof}
    \begin{remark}
        By Lemma \ref{lemma_concave}, it follows that $\overline{\varphi}_1^\prime<0$ in $(0,{D}/{2})$, which means the right side of \eqref{log-concavity} is always positive. Therefore, Theorem \ref{theorem3} establishes an improved log-concavity property for the first positive eigenfunction $\varphi_{1}(x)$ on $\Omega$ given by \eqref{eq:eigenvalue}, strengthening both the log-concavity result of \cite{colesanti2024} and the strong log-concavity result of \cite{Qin2025}.
    \end{remark}

    \section{Maximum principle method}\label{sect3}
    In this section, we prove the lower bound of the spectral gap in Gaussian spaces, which is proven by Andrews and Clutterbuck \cite{andrews2011}. Here we use the elliptic method as in \cite{ni2013} to  give another proof.    
    \begin{theorem}
    Let $\Omega$ be a strictly convex bounded domain in $\mathbb{R}^n$ with diameter $D$. Then the gap between the second eigenvalue $\lambda_{2,G}$ and the first eigenvalue $\lambda_{1,G}$ of $\Delta_G$ satisfies:
    \begin{equation}\label{fundamentalgap}
        \lambda_{2,G} - \lambda_{1,G} \geq \overline{\lambda}_2(D) - \overline{\lambda}_1(D),
    \end{equation}
    where $\overline{\lambda}_1(D)$ and $\overline{\lambda}_2(D)$ are the first two Dirichlet eigenvalues of the one-dimensional model operator \eqref{one_dim_model}. 
    \end{theorem}

    \begin{proof}
    Let $w(x) = \frac{\varphi_2(x)}{\varphi_1(x)}$ and $\overline{w}(s) = \frac{1}{2}\frac{\overline{\varphi}_2(s)}{\overline{\varphi}_1(s)}$. Here $\varphi_i$ are the first two eigenfunctions given by \eqref{equivalent_ep}, and $\overline{\varphi}_i$ are the first two eigenfunctions given by the corresponding one-dimensional model, that is,
    \begin{equation*}
        {-\overline{\varphi}_i''+ \frac{1}{4}s^2\overline{\varphi}_i=\overline{\lambda}_i\overline{\varphi}_i,\quad i=1,2.}
    \end{equation*}
    Then $w$ satisfies
    \begin{equation}\label{e3.2}
        \Delta w = -(\lambda_{2,G} - \lambda_{1,G})w - 2\langle\nabla \log \varphi_1, \nabla w\rangle, 
    \end{equation}
    and $\overline{w}$ satisfies
    \begin{equation}\label{w_one_dim}
        \overline{w}''=-(\overline{\lambda}_2(D) - \overline{\lambda}_1(D))\overline{w} - \frac{2\overline{w}'\overline{\varphi}_1'}{\overline{\varphi}_1}.
    \end{equation}
    Here $\overline{w}(0)=0, \overline{w}'(0)>0$. Without loss of generality, we may assume $\overline{w}'(0)=1$.

    Note that $w$ can be extended to a smooth function on $\overline{\Omega}$ with Neumann boundary condition $\frac{\partial w}{\partial \nu} = 0$ on $\partial\Omega$ \cite{singer1985}, and $\overline{w}$ can also be extended.

    Consider the quotient of the oscillations of $w$ and $\overline{w}(s)$ and let
    \begin{equation}
        Q(x,y) := \frac{w(x) - w(y)}{\overline{w}(\frac{|x-y|}{2})}
    \end{equation}
    on $\bracket{\overline{\Omega} \times \overline{\Omega}} \backslash \Delta$ where $\Delta = \{(x,x) | x \in \overline{\Omega}\}$ is the diagonal. 
    Note that
    \begin{equation*}
        \lim_{y\rightarrow x} Q(x, y) = \frac{2\langle\nabla w(x), X\rangle}{\overline{w}'(0)},
    \end{equation*}
    where $X=\lim_{y\rightarrow x}\bracket{x-y}/\av{x-y}$.
    So the function $Q$ can be extended to the unit sphere bundle $U\overline{\Omega} = \{(x, X) | x \in \overline{\Omega}, |X| = 1\}$, with the extension defined as
    \begin{equation}
        Q(x, X) = \frac{2\langle\nabla w(x), X\rangle}{\overline{w}'(0)}.
    \end{equation}
    Then the maximum of $Q$, which is clearly nonzero and denoted by $m$, is obtained on $\bracket{\overline{\Omega} \times \overline{\Omega}} \backslash \Delta$ or on the unit sphere bundle $U\overline{\Omega}$.

    \textbf{Case 1:} the maximum of $Q$ is attained at some $(x_0, y_0)$ with $x_0 \neq y_0$. Consider the difference of the oscillations of $w$ and $\overline{w}(s)$ and let
    \begin{equation}
        Z(x,y) := {w(x) - w(y)}-m{\overline{w}(\frac{|x-y|}{2})}.
    \end{equation}
    Then $Z(x,y)\leq 0$ and $Z(x_0,y_0)=0$. If $x_0\in \partial\Omega$ and $\nu_{x_0}$ is the outer unit normal vector, then from the Neumann condition $\frac{\partial w}{\partial \nu} = 0$ and the convexity of $\Omega$, we have 
    \begin{align*}
        Z(x_0-t\nu_{x_0},y_0) 
        &= {w(x_0-t\nu_{x_0}) - w(y_0)}-m{\overline{w}(\frac{|x_0-t\nu_{x_0}-y_0|}{2})}\\
        &= w(x_0)- w(y_0) - m\bracket{{\overline{w}(\frac{|x_0-y_0|}{2})}-C_0t{\overline{w}'(\frac{|x_0-y_0|}{2})}} + O(t^2), 
    \end{align*}
    where $C_0=|x_0-y_0|\cos\theta_0$ and $\theta_0$ is the angle between $-\nu_{x_0}$ and $y_0-x_0$. Note that $\overline{w}'>0$ on $(-D/2, D/2)$, which we will prove in Lemma \ref{monotonicity}. Thus, for $t\ll 1$, we get $Z(x_0-t\nu_{x_0},y_0)>Z(x_0,y_0)$, a contradiction, which means that $x_0$ must be in $\Omega$. Similarly, the point $y_0$ should also be in $\Omega$.
    
    Now we pick a local orthonormal frame $\{e_i\}$, $1 \leq i \leq n$ such that $e_n = \frac{y_0 - x_0}{|y_0 - x_0|}$. Since $(x_0,y_0)$ is the maximum point and $\langle e_i, e_n\rangle=0$ for $i=1,\ldots,n-1$, the first variation gives
    \begin{align*}
        &\derivative{}{s}Z(x_0+se_i,y_0+se_i)\Bigg|_{s=0}=\nabla_{e_i}w(x_0) - \nabla_{e_i}w(y_0)=0,\\
        &\derivative{}{s}Z(x_0+se_i,y_0-se_i)\Bigg|_{s=0}=\nabla_{e_i}w(x_0) + \nabla_{e_i}w(y_0)=0.
    \end{align*}
    Also, we have 
    \begin{align*}
        &\derivative{}{s}Z(x_0+se_n,y_0+se_n)\Bigg|_{s=0}=\nabla_{e_n}w(x_0) - \nabla_{e_n}w(y_0)=0,\\
        &\derivative{}{s}Z(x_0+se_n,y_0-se_n)\Bigg|_{s=0}=\nabla_{e_n}w(x_0) + \nabla_{e_n}w(y_0) + m\overline{w}'(\frac{|x_0-y_0|}{2})=0.
    \end{align*}
    Putting them together we have that at $(x_0, y_0)$,
    \begin{equation}\label{e3.3}
        \nabla w(x) = \nabla w(y) = -\frac{m}{2}\overline{w}'(\frac{|x_0-y_0|}{2}) e_n.
    \end{equation}
    From the second variation, we obtain for $1 \leq i \leq n-1$, at $(x_0, y_0)$,
    \begin{equation}\label{e3.4}
    0 \geq \derivative{^2}{s^2}Z(x_0+se_i,y_0+se_i)\Bigg|_{s=0} = \nabla_{e_i,e_i}^2 w(x_0) - \nabla_{e_i,e_i}^2 w(y_0)
    \end{equation}
    and
    \begin{align}\label{e3.5}
        0 &\geq \derivative{^2}{s^2}Z(x_0+se_n,y_0-se_n)\Bigg|_{s=0} \notag\\
        &= {\nabla_{e_n,e_n}^2 w(x) - \nabla_{e_n,e_n}^2 w(y)} - m\overline{w}''(\frac{|x_0-y_0|}{2}).
    \end{align}

    Combining \eqref{e3.4} and \eqref{e3.5}, we can derive
    \begin{equation}\label{e3.6}
        {\Delta w(x) - \Delta w(y)} - m\overline{w}''(\frac{|x_0-y_0|}{2})\leq 0.
    \end{equation}
    Using \eqref{log-concavity}, \eqref{e3.2}, \eqref{w_one_dim} \eqref{e3.3} and noticing that $w(x_0)- w(y_0)=m\overline{w}(\frac{|x_0-y_0|}{2})$, we have that at $(x_0, y_0)$,
    \begin{align}
        0
        &\geq -(\lambda_{2,G} - \lambda_{1,G})m\overline{w}(\frac{|x_0-y_0|}{2}) + 2\as{\nabla \log \varphi_1, \nabla w}\Big|_{x_0}^{y_0} - m\overline{w}''(\frac{|x_0-y_0|}{2}) \notag \\
        &= -(\lambda_{2,G} - \lambda_{1,G})m\overline{w} + m\overline{w}' {\langle -\nabla \log \varphi_1(y_0) + \nabla \log \varphi_1(x_0)\rangle \cdot \frac{y_0-x_0}{|y_0-x_0|}} -m{\overline{w}''}\notag\\
        &\geq -(\lambda_{2,G} - \lambda_{1,G})m\overline{w} - 2m\overline{w}'(\log \overline{\varphi_1})'\Big|_{s=\frac{|x_0-y_0|}{2}} + m{ \left((\overline{\lambda}_2(D) - \overline{\lambda}_1(D))\overline{w} + \frac{2\overline{w}'\overline{\varphi}_1'}{\overline{\varphi}_1}\right)}\notag\\
        &=-(\lambda_{2,G} - \lambda_{1,G})m\overline{w} + (\overline{\lambda}_2(D) - \overline{\lambda}_1(D))m\overline{w},\notag
    \end{align}
    which gives \eqref{fundamentalgap}.

    \textbf{Case 2:} the maximum of $Q$ is attained at some $(x_0, X_0) \in U\overline{\Omega}$. This time the function $Z(x,y)$ is not working since $Z(x,y)$ vanishes at points where $x=y$. So we need to go back to use the function $Q$. The corresponding maximal direction $X_0 = \frac{\nabla w}{|\nabla w|}$ at $x_0$ and $m = 2|\nabla w|(x_0)$. By the definition of $Q$ on $U\overline{\Omega}$, we know that $|\nabla w|(x_0) \geq |\nabla w|(x)$ for any $x \in \overline{\Omega}$. If $x_0\in\partial\Omega$, then by the strict convexity of $\Omega$,
    \begin{equation*}
        0\leq \nabla_\nu|\nabla w|^2 = -\operatorname{II}(\nabla w,\nabla w)<0,
    \end{equation*}
    a contradiction. Here, $\operatorname{II}$ is the second fundamental form of $\partial \Omega$. So the maximum point $x_0\in\Omega$. Now choose an orthonormal frame $\{e_i\}$ at $x_0$ with $e_n = \frac{\nabla w(x_0)}{|\nabla w(x_0)|}$. In this frame, we can see that $w_i = 0$ for $1 \leq i \leq n-1$ and $w_n = \nabla w$.
    
    Since $|\nabla w|(x_0) \geq |\nabla w|(x)$ for any $x \in \overline{\Omega}$, the first derivative of $|\nabla w|^2$ vanishes
    \begin{equation}\label{e3.7}
        \nabla_k|\nabla w|^2 = 2w_{kn}w_n = 0, \quad 1 \leq k \leq n.
    \end{equation}
    And the second derivative of $|\nabla w|^2$ yields for any $1 \leq k \leq n-1$,
    \begin{equation}\label{e3.8}
        0 \geq (|\nabla w|^2)_{kk} = 2\sum_{j=1}^n (w_{jk}^2 + w_{kkj}w_j) \geq 2w_{kkn}w_n.
    \end{equation}

    Let $x(s) = x_0 + se_n$, $y(s) = x_0 - se_n$ and consider $g(s) = Q(x(s), y(s))$. Thus, $g(s) \leq g(0) = m$ for all $s \in (-\epsilon, \epsilon)$, which implies that $\lim_{s \to 0} g'(s) = 0$ and $\lim_{s \to 0} g''(s) \leq 0$.
    Taking derivatives, we have
    \begin{equation}\label{e3.13}
        g'(s) = \frac{\langle\nabla w(x(s)), e_n\rangle + \langle\nabla w(y(s)), e_n\rangle}{\overline{w}} - g(s)\frac{\overline{w}'}{\overline{w}},
    \end{equation}
    and
    \begin{align}\label{e3.14}
        g''(s) &= \frac{w_{nn}(x(s)) - w_{nn}(y(s))}{\overline{w}(s)} - 2\frac{w_{n}(x(s))
        +w_{n}(y(s))}{\overline{w}(s)} \frac{\overline{w}'(s)}{\overline{w}(s)}\\
        &{ {-\frac{\overline{w}''(s)}{\overline{w}(s)}}} g(s) + 2g(s)\left(\frac{\overline{w}'(s)}{\overline{w}(s)}\right)^2.\notag
    \end{align}
    Note that 
    \begin{equation*}
        \lim_{s \to 0} \frac{g'(s)}{\overline{w}(s)} = \frac{g''(0)}{\overline{w}'(0)} = g''(0).
    \end{equation*}
    Combining \eqref{e3.13}, the equation \eqref{e3.14} implies
    \begin{equation*}
        \lim_{s \to 0} g''(s) = 2w_{nnn}(x_0) - 2\lim_{s \to 0} g''(s) { {-\frac{\overline{w}''}{\overline{w}}}} m.
    \end{equation*}
    Therefore, we have that at $x_0$,
    \begin{equation}\label{e3.9}
        2w_{nnn} { {-\frac{\overline{w}''}{\overline{w}}}} m\leq 0. 
    \end{equation}
    Multiplying $w_n$ on \eqref{e3.8} and \eqref{e3.9} and adding them, we can obtain that at $x_0$,
    \begin{equation*}
        0 \geq \langle\nabla(\Delta w), \nabla w\rangle { {-\frac{\overline{w}''}{\overline{w}}}}|\nabla w|^2.
    \end{equation*}
    It follows from the equality \eqref{e3.2} and $w_{kn} = 0$ that at $x_0$ and at $s=0$,
    \begin{align*}
        0 &\geq -(\lambda_{2,G} - \lambda_{1,G})|\nabla w|^2 - 2\nabla^2(\log \varphi_1)(\nabla w, \nabla w) { {-\frac{\overline{w}''}{\overline{w}}}}|\nabla w|^2\\
        &{ \geq -(\lambda_{2,G} - \lambda_{1,G})|\nabla w|^2 - 2(\log \overline{\varphi}_1)''|\nabla w|^2 +\left((\overline{\lambda}_2(D) - \overline{\lambda}_1(D)) + \frac{2\overline{w}'\overline{\varphi}_1'}{\overline{w}\ \overline{\varphi}_1}\right)|\nabla w|^2}\\
        &{> -(\lambda_{2,G} - \lambda_{1,G})|\nabla w|^2 + (\overline{\lambda}_2(D) - \overline{\lambda}_1(D))|\nabla w|^2}.
    \end{align*}
    The last inequality is due to
    \begin{equation*}
        \overline{\varphi}_1>0, \quad \overline{\varphi}_1'(0)=0, \quad \overline{\varphi}_1''(0)<0,
    \end{equation*}
    and
    \begin{equation*}
        (\log \overline{\varphi}_1)''(0)=\frac{\overline{\varphi}_1''\overline{\varphi}_1-\overline{\varphi}_1'^2}{\overline{\varphi}_1^2}(0)=\frac{\overline{\varphi}_1''\overline{\varphi}_1}{\overline{\varphi}_1^2}(0)<0.
    \end{equation*}
    Thus, for both Case 1 and Case 2, the inequality \eqref{fundamentalgap} always holds.
    \end{proof}

    \section{Analysis of fundamental gap}\label{sect4}
    In this section, we study the properties of  the one-dimensional fundamental gap $\bar\lambda_2(D)-\bar\lambda_1(D)$, and prove the sharpness of Theorem \ref{main_theorem}. First, as in \cite{lavine1994}, we give the following lemma.
    \begin{lemma}\label{monotonicity}
        Let $\overline{\varphi}_{1}$ and $\overline{\varphi}_{2}$ be the first two Dirichlet eigenfunctions of \eqref{one_dim_model}. Then $\overline{\varphi}_{1}$ is an even function and $\overline{\varphi}_{2}$ is an odd function. Furthermore, there is some point $b\in(0,D/2)$ such that $\overline{\varphi}_1^2(s)>\overline{\varphi}_2^2(s),\ s\in (0,b)$ and $\overline{\varphi}_1^2(s)<\overline{\varphi}_2^2(s),\ s\in (b,D/2)$. 
    \end{lemma}
    \begin{proof}
        With the gauge transformation $v(s)=u(s)e^{-s^2/4}$, we only need to consider the first two Dirichlet eigenfunctions of the operator
        \begin{equation}\label{after_gauge}
            -\derivative{^2}{s^2} + \frac{1}{4} s^2.
        \end{equation}
        For convenience, we still denote the first two Dirichlet eigenfunctions of \eqref{after_gauge} by $\overline{\varphi}_{1}$ and $\overline{\varphi}_{2}$. 
        
        First, since $s^2$ is an even function, it is easy to see that the first Dirichlet eigenfunction $\overline{\varphi}_{1}$ is an even function. Because the second Dirichlet eigenfunction $\overline{\varphi}_{2}$ is perpendicular to $\overline{\varphi}_{1}$, we know that $\overline{\varphi}_{2}$ is an odd function from the min-max theorem.

        For the second statement, note that $\overline{\varphi}_{2}<0$ in $(-\frac{D}{2},0)$ and $\overline{\varphi}_{2}>0$ in $(0,\frac{D}{2})$. Then for $x\in(-\frac{D}{2},0)$,
        \begin{align*}
            \bracket{\frac{\overline{\varphi}_{2}(x)}{\overline{\varphi}_{1}(x)}}^\prime
            &=\frac{\overline{\varphi}_{2}^\prime(x)\overline{\varphi}_{1}(x)-\overline{\varphi}_{1}^\prime(x)\overline{\varphi}_{2}(x)}{\overline{\varphi}_{1}^2(x)}\\
            &=-\frac{1}{\overline{\varphi}_{1}^2(x)}\int_{-\frac{D}{2}}^{x}(\overline{\lambda}_2(D) -       \overline{\lambda}_1(D))\overline{\varphi}_{1}(s)\overline{\varphi}_{2}(s)ds>0.
        \end{align*}
        Similarly, for $x\in(0,\frac{D}{2})$,
        \begin{align*}
            \bracket{\frac{\overline{\varphi}_{2}(x)}{\overline{\varphi}_{1}(x)}}^\prime
            =\frac{1}{\overline{\varphi}_{1}^2(x)}\int_{0}^{x}(\overline{\lambda}_2(D) -       \overline{\lambda}_1(D))\overline{\varphi}_{1}(s)\overline{\varphi}_{2}(s)ds>0.
        \end{align*}
        Thus, ${\overline{\varphi}_{2}}/{\overline{\varphi}_{1}}$ is increasing in $(-\frac{D}{2},\frac{D}{2})$. We may assume that 
        \begin{equation*}
            \int_{-\frac{D}{2}}^{\frac{D}{2}}\overline{\varphi}_i^2=1,\quad i=1,2.
        \end{equation*}
        Therefore, 
        \begin{equation*}
            \int_{-\frac{D}{2}}^{\frac{D}{2}}(\overline{\varphi}_2^2-\overline{\varphi}_1^2)=0.
        \end{equation*}
        Note that $\overline{\varphi}_2^2/\overline{\varphi}_1^2$ is even in $(-\frac{D}{2},\frac{D}{2})$ and $\overline{\varphi}_2^2/\overline{\varphi}_1^2(0)=0$. Thus, there is some point $b\in(0,D/2)$ such that $\overline{\varphi}_2^2(b)=\overline{\varphi}_1^2(b)$. Then the conclusion follows from the monotonicity of ${\overline{\varphi}_{2}}/{\overline{\varphi}_{1}}$.
    \end{proof}

    With Lemma \ref{monotonicity}, we have the following monotonicity theorem.
    \begin{theorem}\label{theorem_limit_1}
        Let $\overline{\lambda}_1(D)$ and $\overline{\lambda}_2(D)$ be the first two Dirichlet eigenvalues of \eqref{one_dim_model}. Then $(\overline{\lambda}_2(D) - \overline{\lambda}_1(D))D^2$ is increasing in $D$, that is,
        \begin{equation}
            \derivative{\left((\overline{\lambda}_2(D) -       \overline{\lambda}_1(D))D^2\right)}{D}> 0.
        \end{equation}
        Moreover, 
        \begin{equation}\label{limit_D}
            \lim_{D\rightarrow0+} (\overline{\lambda}_2(D) -       \overline{\lambda}_1(D))D^2 =3\pi^2.
        \end{equation}
    \end{theorem}
    \begin{proof}
        Here we only need to consider the first two Dirichlet eigenvalues of \eqref{after_gauge} because of the equality \eqref{difference_eigenvalue}, where we still denote the first two Dirichlet eigenfunctions of \eqref{after_gauge} by $\overline{\varphi}_{1}$ and $\overline{\varphi}_{2}$, the first two Dirichlet eigenvalues of \eqref{after_gauge} by $\overline{\lambda}_1(D)$ and $\overline{\lambda}_2(D)$. 
        
        Consider $\psi_i(s)=\overline{\varphi}_i(Ds)$ for $i=1,2$, which gives
        \begin{equation}\label{scaling}
            \begin{cases} 
            -\psi_i^{\prime\prime}(s) + \frac{1}{4}D^4s^{2}\psi_i(s) = \tilde{\lambda}_i\psi_i(s), & s \in I=(-\frac{1}{2}, \frac{1}{2}), \\ 
            \psi_i(\pm 1/2) = 0,
            \end{cases}
        \end{equation}
        where $\tilde{\lambda}_i=D^2\overline{\lambda}_i$. We may assume that $\int_I\psi_i^2=1$. Thus, it suffices to prove that 
        $$\derivative{(\tilde{\lambda}_2-\tilde{\lambda}_1)}{D}>0.$$

        Recall that $\psi_1$ is even and $\psi_2$ is odd. By Lemma \ref{monotonicity}, we can get some point $b\in(0,1/2)$ such that $\psi_1^2(s)>\psi_2^2(s),\ s\in (0,b)$ and $\psi_1^2(s)<\psi_2^2(s),\ s\in (b,1/2)$ . With the variation of eigenvalues with respect to $D$, we have
        \begin{equation*}
            \derivative{\tilde{\lambda}_i}{D}=\int_{-\frac{1}{2}}^{\frac{1}{2}}\partiald{(\frac{1}{4}D^4s^2)}{D}\psi_i^2(s)d s=2D^3\int_{0}^{\frac{1}{2}}s^2\psi_i^2(s)d s.
        \end{equation*}
            
        Therefore,
        \begin{align*}
            \derivative{(\tilde{\lambda}_2-\tilde{\lambda}_1)}{D}&=2D^3\int_{0}^{\frac{1}{2}}s^2(\psi_2^2(s)-\psi_1^2(s))d s\\
            &=2D^3\int_{0}^{b}s^2(\psi_2^2(s)-\psi_1^2(s))d s+2D^3\int_{b}^{\frac{1}{2}}s^2(\psi_2^2(s)-\psi_1^2(s))d s\\
            &> 2D^3\int_{0}^{b}b^2(\psi_2^2(s)-\psi_1^2(s))d s+2D^3\int_{b}^{\frac{1}{2}}b^2(\psi_2^2(s)-\psi_1^2(s))d s\\
            &= 2D^3b^2\int_{0}^{\frac{1}{2}}(\psi_2^2(s)-\psi_1^2(s))d s=0.
        \end{align*}
            
        For the limit of $\tilde{\lambda}_2-\tilde{\lambda}_1$ when $D$ is near $0$, we know from \eqref{scaling} and the variation method that
        \begin{equation*}
            \tilde{\lambda}_i=\inf_{\substack{E_i\subset C_0^\infty(I)\\ E_i:\ i\text{-dim\ space}}}\sup_{\psi\in E_i} \frac{\int_I \psi^\prime(s)^2+D^4s^2\psi^2ds}{\int_I\psi^2ds}.
        \end{equation*}
        Note that in $I$, $0\leq s^2\leq1/4$. Therefore,
        \begin{equation*}
            \frac{\int_I(\psi^\prime)^2ds}{\int_I\psi^2ds}\leq\frac{\int_I((\psi^\prime)^2+\frac{1}{4}D^4s^2\psi^2)ds}{\int_I\psi^2ds}\leq \frac{\int_I(\psi^\prime)^2ds}{\int_I\psi^2ds}+\frac{D^4}{16}.
        \end{equation*}
        Denote 
        \begin{equation*}
            {\lambda}_i^0=\inf_{\substack{E_i\subset C_0^\infty(I)\\ E_i:\ i\text{-dim\ space}}}\sup_{\psi\in E_i} \frac{\int_I(\psi^\prime)^2ds}{\int_I\psi^2ds}
        \end{equation*}
        and recall that ${\lambda}_1^0=\pi^2$ and ${\lambda}_2^0=4\pi^2$.
        Thus, we obtain
        \begin{equation*}
            {\lambda}_i^0\leq \tilde{\lambda}_i\leq {\lambda}_i^0 + \frac{D^4}{16}
        \end{equation*}
        and then
        \begin{equation*}
            \left|(\tilde{\lambda}_2-\tilde{\lambda}_1)-({\lambda}_2^0-{\lambda}_1^0)\right|\leq\frac{D^4}{16}.
        \end{equation*}
        Hence, we have
        \begin{equation*}
            \lim_{D\rightarrow0+} (\overline{\lambda}_2(D) - \overline{\lambda}_1(D))D^2 ={\lambda}_2^0-{\lambda}_1^0=3\pi^2,
        \end{equation*}
        which completes the proof.
    \end{proof}
    \begin{remark}
        Inequality \eqref{limit_D} tells us that as $D\to 0+$, the normalized gap approaches the gap in the 1-dim Euclidean case, which follows from the observation that the term $D^4s^2$ of \eqref{scaling} disappears as $D\to 0+$.
    \end{remark}

    Besides this, we can show that the inequality \eqref{fundamentalgap} is also sharp for $n\geq2$.
    \begin{theorem}
        Consider the potential $V\equiv 0$ and the thin rectangle
        \begin{equation*}
            \Omega_\varepsilon:=(-\frac{D}{2},\frac{D}{2})\times(-\varepsilon,\varepsilon),
        \end{equation*}
        whose first two eigenvalues of \eqref{equivalent_ep} are denoted by $\lambda^\varepsilon_1$ and $\lambda^\varepsilon_2$. Then we have
        \begin{equation}
            \lim_{\varepsilon\rightarrow0+}(\lambda^\varepsilon_2-\lambda^\varepsilon_1)= \overline{\lambda}_2(D) - \overline{\lambda}_1(D),
        \end{equation}
        where $\overline{\lambda}_1(D)$ and $\overline{\lambda}_2(D)$ be the first two Dirichlet eigenvalues of \eqref{one_dim_model} on $(-\frac{D}{2},\frac{D}{2})$.
    \end{theorem}
    \begin{proof}
        Let $\varphi_{1}(s)$ and $\varphi_{2}(s)$ be the first and second eigenfunctions of \eqref{equivalent_ep} on $\Omega_\varepsilon$ with $V\equiv 0$. As in the proof of Theorem \ref{theorem_limit_1}, we still denote the first two Dirichlet eigenfunctions of \eqref{after_gauge} by $\overline{\varphi}_{1}$ and $\overline{\varphi}_{2}$, the first two Dirichlet eigenvalues of \eqref{after_gauge} by $\overline{\lambda}_1(D)$ and $\overline{\lambda}_2(D)$. Define 
        \begin{equation*}
            u_1(s,t)=\overline{\varphi}_{1}(s)\cos\bracket{\frac{\pi t}{2\varepsilon}},\quad u_2(s,t)=\overline{\varphi}_{2}(s)\cos\bracket{\frac{\pi t}{2\varepsilon}}.
        \end{equation*}
        Direct calculation shows that $u_1$ satisfies
        \begin{equation*}
            -\Delta u_1(s,t)= \bracket{\overline{\lambda}_1(D)-s^2+\frac{\pi^2}{4\varepsilon^2}}u_1,
        \end{equation*}
        Therefore, with a standard argument,
        \begin{align*}
            \int_{\Omega_\varepsilon}\bracket{\lambda^\varepsilon_1-\overline{\lambda}_1(D)-\frac{\pi^2}{4\varepsilon^2}}u_1\varphi_1
            &= -\int_{\Omega_\varepsilon}u_1\Delta\varphi_1+\int_{\Omega_\varepsilon}\varphi_1(\Delta u_1+t^2u_1)\\
            &\geq -\int_{\Omega_\varepsilon}u_1\Delta\varphi_1+\int_{\Omega_\varepsilon}\varphi_1\Delta u_1=0.
            \end{align*}
        which means that
        \begin{equation*}
            \lambda^\varepsilon_1\geq \overline{\lambda}_1(D)+\frac{\pi^2}{4\varepsilon^2}.
        \end{equation*}
        Note that $\overline{\varphi}_{2}(s)$ is odd and ${\varphi}_{1}(s,t)$ is even with respect to $s$. Therefore,
        \begin{equation*}
            \int_{\Omega_\varepsilon}u_2{\varphi}_{1}=0.
        \end{equation*}
        Similarly, for $u_2$, we have
        \begin{equation*}
            -u_2\Delta u_2(s,t)= \bracket{\overline{\lambda}_2(D)-s^2+\frac{\pi^2}{4\varepsilon^2}}u_2^2\leq \bracket{\overline{\lambda}_2(D)+\varepsilon^2+\frac{\pi^2}{4\varepsilon^2}-s^2-t^2}u_2^2.
        \end{equation*}
        Thus,
        \begin{equation*}
            \lambda^\varepsilon_2\leq \frac{\int_{\Omega_\varepsilon}\bracket{\av{\nabla u_2}^2+(s^2+t^2)u_2^2}dsdt}{\int_{\Omega_\varepsilon}u_2^2dsdt}\leq \overline{\lambda}_2(D)+\varepsilon^2+\frac{\pi^2}{4\varepsilon^2}.
        \end{equation*}
        Hence, 
        \begin{equation*}
            \lambda^\varepsilon_2-\lambda^\varepsilon_1\leq \overline{\lambda}_2(D) - \overline{\lambda}_1(D)+\varepsilon^2,
        \end{equation*}
        which completes the proof.
    \end{proof}
    
% Though the diameter $D$ of $\Omega$ in all theorems refers to the Euclidean metric. However, it is more natural to consider the diameter  $D_G$ of $\Omega$ under the Gaussian metric. This leads to the following question:
% \begin{question}
%     Does the fundamental gap estimate \eqref{fundamental} remain valid if the diameter $D$ is replaced by $D_G$?
% \end{question}

\section{Log-concavity of Gaussian heat kernel on convex domains}\label{sect5}

In this section, we prove Theorem \ref{theorem-heatkernel}, which establishes the improved log-concavity of the Gaussian heat kernel on convex domains in Gaussian space. 

Let the Gaussian heat kernel $H_G(x,y,t)$ and $\overline{H}_G(s,t)$ be given in Theorem \ref{theorem-heatkernel}. Let $M(x,y,t)$ be the Mehler kernel given by
\begin{equation*}
    M(x,y,t):=\frac{1}{\bracket{1-e^{-2t}}^{{n}/{2}}}\exp\bracket{-\frac{1}{2}\frac{e^{-2t}(\av{x}^2+\av{y}^2)-2e^{-t}x\cdot y}{1-e^{-2t}}},
\end{equation*}
which satisfies $\partial_t M=L_\mu M$ in $\R^n$.
And let $\overline{M}(s,t)$ be the corresponding fundamental solution for the operator $\partial_t-\partial_s^2+s\partial_s$, concentrating at $s=0$. That is,
\begin{equation}\label{mehler_one_dim}
    \overline{M}(s,t):=\frac{1}{\bracket{1-e^{-2t}}^{{n}/{2}}}\exp\bracket{-\frac{1}{2}\frac{e^{-2t}s^2}{1-e^{-2t}}}.
\end{equation}

Define $\Phi(x,t):=\bracket{H_G/M}(z,x,t)$ for any fixed point $z\in\Omega$, and $\psi(s,t):=\bracket{\overline{H}_G/\overline{M}}(s,t)$. Note that from symmetry considerations, we have $$\bracket{\log\overline{H}_G}^\prime(0,t)=\bracket{\log\overline{M}}^\prime(0,t)=0.$$

Our first step is to establish the log-concavity of  $\overline{H}_G(s,t)$ and $\psi(s,t)$ as functions of the spatial variable $s$. This property is fundamental and will be utilized in the subsequent proof of the improved log-concavity estimate.

% from the result of "H. Brascamp, E. Lieb, On the extensions of the Brunn–Minkowski and Prékopa–Leindler theorems, including inequalities for log concave functions, and with an application to the diffusion equation, J. Functional Analysis 22 (1976) 366–389."
\begin{lemma}\label{log-concavity-1-dim}
    The functions $\overline{H}_G(s,t)$ and $\psi(s,t)$ are log-concave in $s$. In particular, $\bracket{\log\overline{H}_G}^\prime<0$ and $\bracket{\log\psi}^\prime<0$ on $(0,D/2)$.
\end{lemma}
\begin{proof}
    It suffices to show that $\psi(s,t)$ is log-concave, since the log-concavity of $\overline{H}_G(s,t)$ follows from the identity $\log\overline{H}_G(s,t)=\log\psi(s,t)+\log\overline{M}(s,t)$, where $\overline{M}(s,t)$ is log-concave for all $t>0$.
    
    To prove the log-concavity of $\psi(s,t)$, we consider the gauge transformations $\widetilde{H}_G(s,t):=\overline{H}_G(s,t)e^{-s^2/4}$ and $\widetilde{M}(s,t):=\overline{M}(s,t)e^{-s^2/4}$. Then $\widetilde{H}_G(s,t)$ is the fundamental solution of the following equation:
    \begin{equation}
    \begin{cases}
        \partial_t u(s,t)-\partial_s^2u(s,t)+(\frac{s^2}{4}-\frac{1}{2})u(s,t)=0,\quad s\in I_D:=(-\frac{D}{2},\frac{D}{2}).\\
        u(-\frac{D}{2},t) = u(\frac{D}{2},t) =  0.
    \end{cases}
    \end{equation}
    
    By applying the Trotter product formula with  $x_0 = 0, x_N = s$, we obtain
    \begin{align}\label{expansion_for_tilde_heatkernel}
        \widetilde{H}_G(s,t) =& \lim_{N \to \infty} (4\pi t/N)^{-N/2} \int_{\mathbb{R}^{N-1}} dx_1 \ldots dx_{N-1} \\
        &\times \prod_{j=1}^N \left\{ \exp \left[ -\frac{N}{4t}(x_j - x_{j-1})^2 - \frac{4t}{N}V(x_j) \right] \chi_{I_D}(x_j) \right\}\notag,
    \end{align}
    where $V(s)={s^2}/{4}-{1}/{2}$. By direct calculation, we obtain from equation \eqref{mehler_one_dim} that
    \begin{equation}\label{tilde_mehler_1_dim}
        \widetilde{M}(s,t)=\frac{1}{\bracket{1-e^{-2t}}^{{n}/{2}}}\exp\bracket{-\frac{s^2}{4}\frac{1+e^{-2t}}{1-e^{-2t}}}=\frac{\exp\bracket{-\frac{s^2}{4}\coth{t}}}{\bracket{1-e^{-2t}}^{{n}/{2}}}.
    \end{equation}
    
    Hence, from equations \eqref{expansion_for_tilde_heatkernel}, \eqref{tilde_mehler_1_dim}, together with $\psi(s,t)={\widetilde{H}_G(s,t)}/{\widetilde{M}(s,t)}$, we have
    \begin{align}\label{expansion_for_quotient}
        &\psi(s,t)
        =\lim_{N \to \infty} \bracket{1-e^{-2t}}^{-{n}/{2}}(4\pi t/N)^{-N/2} \int_{\mathbb{R}^{N-1}} dx_1 \ldots dx_{N-1}\\
        &\times \left\{ \exp \left[ -\frac{N}{4t}\sum_{j=1}^N(x_j - x_{j-1})^2 - \frac{4t}{N}\sum_{j=1}^NV(x_j) + \frac{s^2}{4}\coth{t} \right] \chi_{I_D}(x_j) \right\}\notag.
    \end{align}
    
    To establish log-concavity, it suffices to show the integrand of \eqref{expansion_for_quotient} is log-concave by Theorem 1.1 in \cite{Brascamp1975some}. So we only need to prove the following inequality for $N\gg 1$ with $x_0=0$ and any $x_i\in I_D$ for $i=1,\ldots,N$.
    \begin{equation}\label{quadradic_form}
        F_t(x_1,\ldots,x_N)=\frac{N}{t} \sum_{j=1}^{N} (x_j - x_{j-1})^2 + \frac{4t}{N} \sum_{j=1}^{N} x_j^2 - x_N^2 \coth(t)\geq 0.
    \end{equation}
    The inequality is trivial when $x_N=0$. Therefore, after scaling, we may assume the largest point is $x_{j_0}=1$ for some $1\leq j_0\leq N$ and $\av{x_j}\leq 1$ for any $1\leq j\leq N$. Let $\{y_k\}_{k=1}^{N}=\{\av{x_j}\}_{j=1}^{N}$ be the sequence obtained by taking absolute values and arranged so that $0\leq y_1\leq\cdots\leq y_{N-1}\leq y_N=1$.

    We claim that $F_t(y_1,\ldots,y_N)\leq F_t(x_1,\ldots,x_N)$.
    
    Indeed, after the absolute value operation, it is straightforward to verify that the first term $\sum_{j=1}^{N} (x_j - x_{j-1})^2$ in \eqref{quadradic_form} becomes smaller while the other terms remain unchanged. Subsequently, rearranging the sequence to satisfy $y_1\leq\cdots\leq y_{N_1}\leq y_N=1$ further decreases both the first term and the last term in \eqref{quadradic_form}. Consequently, we obtain the sequence $\{y_k\}_{k=1}^{N}$ and $F_t(y_1,\ldots,y_N)\leq F_t(x_1,\ldots,x_N)$.

    Therefore, it suffices to establish inequality \eqref{quadradic_form} under the condition $0\leq x_1\leq\cdots\leq x_{N-1}\leq x_N=1$. We define the piecewise linear function $y(s)$ for $s \in [0,1]$ by:
    \begin{equation}
        y\left( s \right) = x_j + (x_j - x_{j-1})\bracket{Ns-j}, \quad \text{for}\ s\in [\frac{j-1}{N},\frac{j}{N}],\ j=1,\ldots,N.
    \end{equation}
    Then $y(s)$ is a Lipschitz function satisfying
    \begin{align*}
        &\sum_{j=1}^{N} (x_j - x_{j-1})^2 = \frac{1}{N} \int_0^1 [y'(s)]^2  ds, \\
        &\sum_{j=0}^{N-1} x_j^2 \leq N \int_0^1 [y(s)]^2  ds 
        \leq \sum_{j=1}^{N} x_j^2.
    \end{align*}
    
    We define the continuous functional $I(y)$ as
    \begin{equation*}
        I(y) = \frac{1}{t} \int_0^1 [y'(s)]^2  ds + 4t \int_0^1 [y(s)]^2  ds.
    \end{equation*}
    Thus, we obtain 
    \begin{equation}\label{perturbation}
        \av{F_t-I(y)+\coth(t)}\leq \frac{4t}{N}.
    \end{equation}
    
    By the Euler--Lagrange equation, the minimizer $\bar{y}$ of $I[y]$ subject to $y(0) = 0$ and $y(1) = 1$ satisfies:
    \begin{equation*}
        \bar{y}^{\prime\prime} - 4t^2 \bar{y} = 0.
    \end{equation*}
    The general solution is $\bar{y}(s) = A e^{2ts} + B e^{-2ts}$. The initial conditions $y(0) = 0, y(1)=1$ yield $A=-B=1/(2\sinh (2t))$. Therefore, the extremal function is
    \begin{equation*}
        \bar{y}(s) = \frac{\sinh(2ts)}{\sinh(2t)}.
    \end{equation*}
    
    Direct calculation gives
    \begin{align*}
        I[\bar{y}] &= \frac{1}{t} \int_0^1 \frac{4t^2 \cosh^2(2ts)}{\sinh^2(2t)}  ds + 4t \int_0^1 \frac{\sinh^2(2ts)}{\sinh^2(2t)}  ds \\
        &= \frac{4t}{\sinh^2(2t)} \int_0^1 \cosh(4ts)  ds\\
        &= \coth(t) + \tanh(t).
    \end{align*}
    Here we have used the identities:
    \begin{align*}
        \sinh(4t) = 2 \sinh(2t) \cosh(2t), \quad \text{and}\quad 2 \coth(2t) = \coth(t) + \tanh(t).
    \end{align*}
    
    Therefore, by inequalities \eqref{perturbation} and $I(\bar{y})\leq I(y)$, we obtain the desired inequality
    \begin{equation*}
        F_t\geq I(y)-\coth(t)-\frac{4t}{N}\geq I(\bar{y})-\coth(t)-\frac{4t}{N} = \tanh(t)-\frac{4t}{N},
    \end{equation*}
    which means $F_t>0$ for $N>4t/\tanh(t)$. This demonstrates that the integrand of \eqref{expansion_for_quotient} is log-concave, which completes the proof.
\end{proof}

Let $\Psi(s,t):=(\log\psi(s,t))^\prime=\bracket{\log\bracket{\overline{H}_G/\overline{M}}(s,t)}^\prime$. By direct calculations, we obtain the following equation satisfied by $\Psi$.
\begin{equation*}
    \Psi_t - \Psi'' = 2\Psi\Psi' - {\Psi}\coth{t} - s{\Psi'}\coth{t}.
\end{equation*}
For the vector field $X(x,t) = -\nabla \log \Phi(x,t)$, we have the evolution equation
\begin{align*}
    \left(\frac{\partial}{\partial t} - \Delta\right)X(x,t) = W(x) - 2\nabla_X X(x,t) - &\langle\nabla_{(\cdot)}X(x,t), x-\frac{z}{\cosh{t}}\rangle\coth{t} \\
    & - X(x,t)\coth{t}.
\end{align*}
Here $\langle\nabla_{(\cdot)}X(x,t), x-\frac{z}{\cosh{t}}\rangle$ denotes a vector whose inner product with any vector $Y$ is $\langle\nabla_{Y}X(x,t), x-\frac{z}{\cosh{t}}\rangle$. If the operator in \eqref{eq:GaussHeatKernel} is $L_\mu-V$, then $W(x)=\nabla{V(x)}$, where $V(x)$ is assumed to be convex. Under the condition of Theorem \ref{theorem-heatkernel}, we may simply assume $V=0$.

Therefore, with a similar argument of Theorem 4.2 of Ni \cite{ni2013}, we can derive the following lemma, which is the key step in proving the improved log-concavity estimate.
\begin{lemma}\label{parabolic-maximum-principle}
    With the notation above, the function
    \begin{equation*}
        \mathcal{C}(x,y,t) := (\tanh t)\left(\left\langle X(y,t) - X(x,t), \frac{y-x}{|y-x|}\right\rangle + 2\Psi\left(\frac{|y-x|}{2},t\right)\right)
    \end{equation*}
    cannot attain a negative minimum in the parabolic interior of $\Omega\times\Omega\times(0,T]$ for any fixed $T>0$.
\end{lemma}

The proof of Lemma \ref{parabolic-maximum-principle} follows from a careful application of the maximum principle to the parabolic operator $\partial_t - \Delta$ applied to $C(x,y,t)$, utilizing the log-concavity of $\Psi$ established in Lemma \ref{log-concavity-1-dim} and the convexity of the domain $\Omega$. The details are similar to those in \cite{ni2013} and are omitted here.

With Lemma \ref{log-concavity-1-dim} and Lemma \ref{parabolic-maximum-principle}, we can now prove the main result on the log-concavity of the Gaussian heat kernel.
\begin{theorem}
    With the notation above, we have
    \begin{equation}\label{log-concavity-quotient}
            -\left\langle \nabla_y \log \Phi(y,t) - \nabla_x \log \Phi(x,t), \frac{y-x}{|y-x|} \right\rangle \geq -2\left(\log\Psi\right)^\prime\left(\frac{|y-x|}{2},t\right). 
    \end{equation}
\end{theorem}
\begin{proof}
The same argument as in Section 3 in \cite{ni2013} implies that near the boundary, $C(x,y,t)\geq 0$. As $t\rightarrow 0+$, using the gauge transformation, we only need to consider the case when $H_G$ is the Dirichlet heat kernel of the operator $\partial_t-\Delta+q$, where $q=V+{\av{x}^2}/{4}-n/2$ is a convex function. Therefore, the short-time behavior of the logarithmic derivatives of the heat kernel \cites{Malliavin1996,Neel2007} implies that $C(x,y,t)\geq 0$ holds as $t\rightarrow 0+$. Hence, Lemma \ref{parabolic-maximum-principle} implies that $C(x,y,t)\geq 0$, which completes the proof.
\end{proof}

\begin{remark}
    Note that
    \begin{equation*}
        \left\langle \nabla_y \log M(z,y,t) - \nabla_x \log M(z,x,t), \frac{y-x}{|y-x|} \right\rangle = 2\left(\log\overline{M}\right)^\prime\left(\frac{|y-x|}{2},t\right).
    \end{equation*}
    Therefore, the inequality \eqref{log-concavity-quotient} is exactly equivalent to the inequality \eqref{log-concavity-heatkernel}, which completes the proof of Theorem \ref{theorem-heatkernel}.
\end{remark}

\section*{Acknowledgments}
We thank Professors Julie Clutterbuck, Lei Ni, Shoo Seto, Guofang Wei for helpful discussions and insightful comments. We also thank Yifeng Meng for his sincere help. The second author is supported by NSF of Jiangsu Province No. BK20231309.

    \section*{Appendix}
    We define the normalized gap as
    \begin{equation*}
        \frac{(\overline{\lambda}_2(D) - \overline{\lambda}_1(D))D^2 }{3\pi^2}
    \end{equation*}
    
    In this section, we present numerically computed values of the normalized gap for various diameters $D$. The computations were performed using the Chebyshev spectral collocation method applied to the operator given in \eqref{scaling}, namely,
    \begin{equation*}
        -\derivative{^2}{s^2} + \frac{1}{4}D^4s^2.
    \end{equation*}
    
    As shown in the graph below, the normalized gap increases with $D$. Furthermore, in accordance with Theorem~\ref{theorem_limit_1}, the normalized gap approaches $1$ as $D$ tends to $0$.
    \begin{figure}[H]
        \centering
        \includegraphics[width=0.75\textwidth]{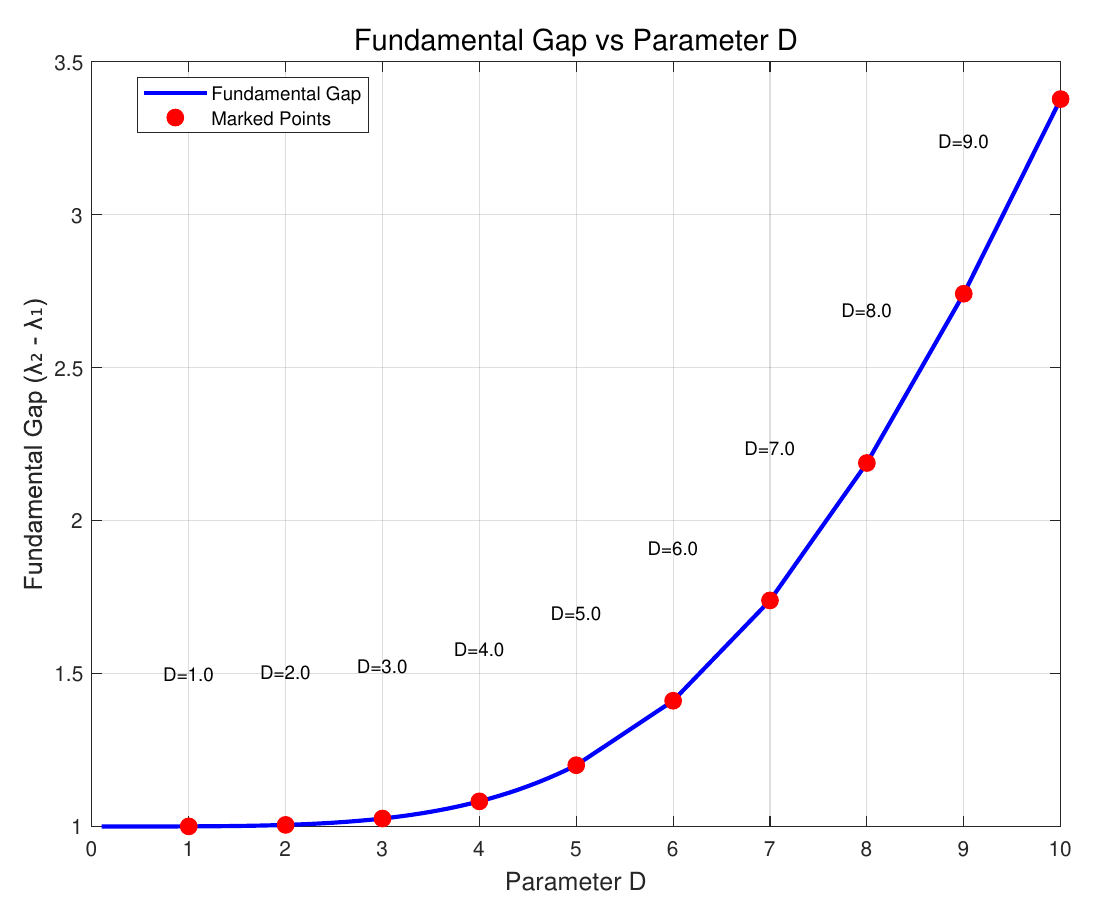}
        \caption{The fundamental gap $\frac{(\overline{\lambda}_2(D) - \overline{\lambda}_1(D))D^2 }{3\pi^2}$.}
        \label{fig:fundamental_gap}
    \end{figure}
    
    The results indicate that the gap increases with $D$, though not linearly. Detailed values are provided in the table below.
    \begin{table}[H]
        \centering
        \begin{tabular}{
          S[table-format=2.1]
          S[table-format=3.6]
          S[table-format=3.6]
          S[table-format=1.6]
        }
        \toprule
        {$D$} & {$\overline{\lambda}_1(D)$} & {$\overline{\lambda}_2(D)$} & {Gap} \\
        \midrule
        1.0 & 9.877771  & 39.496084 & 1.000321 \\
        2.0 & 10.000000 & 39.760812 & 1.005134 \\
        3.0 & 10.523736 & 40.902321 & 1.025998 \\
        4.0 & 11.887886 & 43.930465 & 1.082197 \\
        5.0 & 14.565218 & 50.105109 & 1.200315 \\
        6.0 & 18.862067 & 60.642110 & 1.411068 \\
        7.0 & 24.771932 & 76.264955 & 1.739111 \\
       8.0 & 32.063557 & 96.863410 & 2.188533 \\
       9.0 & 40.511000 & 121.695967 & 2.741919 \\
       10.0 & 50.001421 & 150.032187 & 3.378412 \\
        \bottomrule
        \end{tabular}
        \caption{Eigenvalues ($\overline{\lambda}_1(D)$, $\overline{\lambda}_2(D)$), Normalized gap (Gap)}
        \label{tab:eigenvalues}
    \end{table}

	\bibliography{Fundamental_Gap_on_Gaussian_Spaces}              %Fundamental Gap on Gaussian Spaces为.bib文件名

\begin{bibdiv}
\begin{biblist}

\bib{And15}{incollection}{
      author={Andrews, Ben},
       title={Moduli of continuity, isoperimetric profiles, and multi-point estimates in geometric heat equations},
        date={2015},
   booktitle={Surveys in differential geometry 2014. {R}egularity and evolution of nonlinear equations},
      series={Surv. Differ. Geom.},
      volume={19},
   publisher={Int. Press, Somerville, MA},
       pages={1\ndash 47},
         url={https://doi.org/10.4310/SDG.2014.v19.n1.a1},
      review={\MR{3381494}},
}

\bib{andrews2011}{article}{
      author={Andrews, Ben},
      author={Clutterbuck, Julie},
       title={Proof of the fundamental gap conjecture},
        date={2011},
        ISSN={0894-0347,1088-6834},
     journal={J. Amer. Math. Soc.},
      volume={24},
      number={3},
       pages={899\ndash 916},
         url={https://doi.org/10.1090/S0894-0347-2011-00699-1},
      review={\MR{2784332}},
}

\bib{ACH20}{article}{
      author={Andrews, Ben},
      author={Clutterbuck, Julie},
      author={Hauer, Daniel},
       title={Non-concavity of the {R}obin ground state},
        date={2020},
        ISSN={2168-0930},
     journal={Camb. J. Math.},
      volume={8},
      number={2},
       pages={243\ndash 310},
         url={https://doi.org/10.4310/cjm.2020.v8.n2.a1},
      review={\MR{4091026}},
}

\bib{ACH21}{article}{
      author={Andrews, Ben},
      author={Clutterbuck, Julie},
      author={Hauer, Daniel},
       title={The fundamental gap for a one-dimensional {S}chr\"{o}dinger operator with {R}obin boundary conditions},
        date={2021},
        ISSN={0002-9939},
     journal={Proc. Amer. Math. Soc.},
      volume={149},
      number={4},
       pages={1481\ndash 1493},
         url={https://doi.org/10.1090/proc/15140},
      review={\MR{4242306}},
}

\bib{ashbaugh1989}{article}{
      author={Ashbaugh, Mark~S.},
      author={Benguria, Rafael},
       title={Optimal lower bound for the gap between the first two eigenvalues of one-dimensional {S}chr\"odinger operators with symmetric single-well potentials},
        date={1989},
        ISSN={0002-9939,1088-6826},
     journal={Proc. Amer. Math. Soc.},
      volume={105},
      number={2},
       pages={419\ndash 424},
         url={https://doi.org/10.2307/2046959},
      review={\MR{942630}},
}

\bib{BCNSWW22}{article}{
      author={Bourni, Theodora},
      author={Clutterbuck, Julie},
      author={Nguyen, Xuan~Hien},
      author={Stancu, Alina},
      author={Wei, Guofang},
      author={Wheeler, Valentina-Mira},
       title={The vanishing of the fundamental gap of convex domains in {$\Bbb {H}^n$}},
        date={2022},
        ISSN={1424-0637},
     journal={Ann. Henri Poincar\'{e}},
      volume={23},
      number={2},
       pages={595\ndash 614},
         url={https://doi.org/10.1007/s00023-021-01096-3},
      review={\MR{4386445}},
}

\bib{Brascamp1975some}{incollection}{
      author={Brascamp, H.~J.},
      author={Lieb, Elliott~H.},
       title={Some inequalities for gaussian measures and the long-range order of the one-dimensional plasma},
        date={1975},
   booktitle={Functional integration and its applications},
      editor={Arthurs, A.~M.},
   publisher={Clarendon Press},
     address={Oxford},
       pages={1\ndash 14},
}

\bib{BrascampLieb1976}{article}{
      author={Brascamp, Herm~Jan},
      author={Lieb, Elliott~H.},
       title={On extensions of the {B}runn-{M}inkowski and {P}r\'ekopa-{L}eindler theorems, including inequalities for log concave functions, and with an application to the diffusion equation},
        date={1976},
     journal={J. Functional Analysis},
      volume={22},
      number={4},
       pages={366\ndash 389},
      review={\MR{0450480}},
}

\bib{carlen2001}{article}{
      author={Carlen, E.~A.},
      author={Kerce, C.},
       title={On the cases of equality in {B}obkov's inequality and {G}aussian rearrangement},
        date={2001},
        ISSN={0944-2669,1432-0835},
     journal={Calc. Var. Partial Differential Equations},
      volume={13},
      number={1},
       pages={1\ndash 18},
         url={https://doi.org/10.1007/PL00009921},
      review={\MR{1854254}},
}

\bib{CWY25}{article}{
      author={Cho, Gunhee},
      author={Wei, Guofang},
      author={Yang, Guang},
       title={Probabilistic method to fundamental gap problems on the sphere},
        date={2025},
        ISSN={0002-9947},
     journal={Trans. Amer. Math. Soc.},
      volume={378},
      number={1},
       pages={317\ndash 337},
         url={https://doi.org/10.1090/tran/9285},
      review={\MR{4840306}},
}

\bib{colesanti2024}{article}{
      author={Colesanti, Andrea},
      author={Francini, Elisa},
      author={Livshyts, Galyna},
      author={Salani, Paolo},
       title={The brunn-minkowski inequality for the first eigenvalue of the ornstein-uhlenbeck operator and log-concavity of the relevant eigenfunction},
        date={2024},
     journal={arXiv preprint arXiv:2407.21354},
}

\bib{DSW21}{article}{
      author={Dai, Xianzhe},
      author={Seto, Shoo},
      author={Wei, Guofang},
       title={Fundamental gap estimate for convex domains on sphere---the case {$n = 2$}},
        date={2021},
        ISSN={1019-8385},
     journal={Comm. Anal. Geom.},
      volume={29},
      number={5},
       pages={1095\ndash 1125},
         url={https://doi.org/10.4310/CAG.2021.v29.n5.a3},
      review={\MR{4349140}},
}

\bib{ehrhard1984}{article}{
      author={Ehrhard, Antoine},
       title={In\'egalit\'es isop\'erim\'etriques et int\'egrales de {D}irichlet gaussiennes},
        date={1984},
        ISSN={0012-9593},
     journal={Ann. Sci. \'Ecole Norm. Sup. (4)},
      volume={17},
      number={2},
       pages={317\ndash 332},
         url={http://www.numdam.org/item?id=ASENS_1984_4_17_2_317_0},
      review={\MR{760680}},
}

\bib{evans2010}{book}{
      author={Evans, Lawrence~C.},
       title={Partial differential equations},
     edition={Second},
      series={Graduate Studies in Mathematics},
   publisher={American Mathematical Society, Providence, RI},
        date={2010},
      volume={19},
        ISBN={978-0-8218-4974-3},
         url={https://doi.org/10.1090/gsm/019},
      review={\MR{2597943}},
}

\bib{HWZ20}{article}{
      author={He, Chenxu},
      author={Wei, Guofang},
      author={Zhang, Qi~S.},
       title={Fundamental gap of convex domains in the spheres},
        date={2020},
        ISSN={0002-9327},
     journal={Amer. J. Math.},
      volume={142},
      number={4},
       pages={1161\ndash 1191},
         url={https://doi.org/10.1353/ajm.2020.0033},
      review={\MR{4124117}},
}

\bib{hermite1864}{book}{
      author={Hermite, M},
       title={Sur un nouveau d{\'e}veloppement en s{\'e}rie des fonctions},
   publisher={Imprimerie de Gauthier-Villars},
        date={1864},
}

\bib{KNTW25}{article}{
      author={Khan, Gabriel},
      author={Nguyen, Xuan~Hien},
      author={Tuerkoen, Malik},
      author={Wei, Guofang},
       title={Log-concavity and fundamental gaps on surfaces of positive curvature},
        date={2025},
        ISSN={1019-8385},
     journal={Comm. Anal. Geom.},
      volume={33},
      number={1},
       pages={239\ndash 260},
         url={https://doi.org/10.4310/cag.250221035055},
      review={\MR{4870313}},
}

\bib{lavine1994}{article}{
      author={Lavine, Richard},
       title={The eigenvalue gap for one-dimensional convex potentials},
        date={1994},
        ISSN={0002-9939,1088-6826},
     journal={Proc. Amer. Math. Soc.},
      volume={121},
      number={3},
       pages={815\ndash 821},
         url={https://doi.org/10.2307/2160281},
      review={\MR{1185270}},
}

\bib{Malliavin1996}{article}{
      author={Malliavin, Paul},
      author={Stroock, Daniel~W.},
       title={Short time behavior of the heat kernel and its logarithmic derivatives},
        date={1996},
        ISSN={0022-040X,1945-743X},
     journal={J. Differential Geom.},
      volume={44},
      number={3},
       pages={550\ndash 570},
         url={http://projecteuclid.org/euclid.jdg/1214459221},
      review={\MR{1431005}},
}

\bib{Neel2007}{article}{
      author={Neel, Robert},
       title={The small-time asymptotics of the heat kernel at the cut locus},
        date={2007},
        ISSN={1019-8385,1944-9992},
     journal={Comm. Anal. Geom.},
      volume={15},
      number={4},
       pages={845\ndash 890},
         url={https://doi.org/10.4310/cag.2007.v15.n4.a7},
      review={\MR{2395259}},
}

\bib{ni2013}{article}{
      author={Ni, Lei},
       title={Estimates on the modulus of expansion for vector fields solving nonlinear equations},
        date={2013},
        ISSN={0021-7824,1776-3371},
     journal={J. Math. Pures Appl. (9)},
      volume={99},
      number={1},
       pages={1\ndash 16},
         url={https://doi.org/10.1016/j.matpur.2012.05.009},
      review={\MR{3003280}},
}

\bib{Qin2025}{article}{
      author={Qin, Lei},
       title={The strong log-concavity for first eigenfunction of the {O}rnstein-{U}hlenbeck operator in the class of convex bodies},
        date={2025},
     journal={arXiv preprint arXiv:2507.00819},
}

\bib{SWW19}{article}{
      author={Seto, Shoo},
      author={Wang, Lili},
      author={Wei, Guofang},
       title={Sharp fundamental gap estimate on convex domains of sphere},
        date={2019},
        ISSN={0022-040X},
     journal={J. Differential Geom.},
      volume={112},
      number={2},
       pages={347\ndash 389},
         url={https://doi.org/10.4310/jdg/1559786428},
      review={\MR{3960269}},
}

\bib{singer1985}{article}{
      author={Singer, I.~M.},
      author={Wong, Bun},
      author={Yau, Shing-Tung},
      author={Yau, Stephen S.-T.},
       title={An estimate of the gap of the first two eigenvalues in the {S}chr\"odinger operator},
        date={1985},
        ISSN={0391-173X,2036-2145},
     journal={Ann. Scuola Norm. Sup. Pisa Cl. Sci. (4)},
      volume={12},
      number={2},
       pages={319\ndash 333},
         url={http://www.numdam.org/item?id=ASNSP_1985_4_12_2_319_0},
      review={\MR{829055}},
}

\bib{vandenberg1983}{article}{
      author={van~den Berg, M.},
       title={On condensation in the free-boson gas and the spectrum of the {L}aplacian},
        date={1983},
        ISSN={0022-4715,1572-9613},
     journal={J. Statist. Phys.},
      volume={31},
      number={3},
       pages={623\ndash 637},
         url={https://doi.org/10.1007/BF01019501},
      review={\MR{711491}},
}

\bib{yau1986}{article}{
      author={Yau, Shing~Tung},
       title={Nonlinear analysis in geometry},
    language={ger},
        date={1986},
        ISSN={0013-8584},
     journal={Enseignement mathématique},
      volume={33},
      number={1-2},
       pages={109},
}

\bib{yu1986}{article}{
      author={Yu, Qi~Huang},
      author={Zhong, Jia~Qing},
       title={Lower bounds of the gap between the first and second eigenvalues of the {S}chr\"odinger operator},
        date={1986},
        ISSN={0002-9947,1088-6850},
     journal={Trans. Amer. Math. Soc.},
      volume={294},
      number={1},
       pages={341\ndash 349},
         url={https://doi.org/10.2307/2000135},
      review={\MR{819952}},
}

\end{biblist}
\end{bibdiv}

\end{document}